\documentclass[11pt]{article}
\usepackage[text={6.5in,8in},centering]{geometry}
\usepackage{graphicx,url,subfig}
\RequirePackage[OT1]{fontenc}
\RequirePackage{amsthm,amsmath,amssymb,amscd,dsfont}
\RequirePackage[numbers]{natbib}
\RequirePackage[colorlinks,citecolor=blue,linkcolor = blue, urlcolor=blue]{hyperref}
\usepackage{enumerate, comment}
\usepackage{datetime}
\usepackage{etex}

\usepackage{mathrsfs}
\usepackage{amsfonts}
\usepackage{latexsym,bm,amsfonts,amssymb,pifont,mathbbol}
\usepackage{amsmath}
\usepackage{color}

\usepackage{showlabels}
\usepackage{enumerate}

\makeatother
\numberwithin{equation}{section}
\allowdisplaybreaks

\usepackage{enumerate}
\newcommand{\R}{\mathbb{R}}
\newcommand{\W}{\mathbb{W}}
\newcommand{\HH}{\mathfrak{H}}

\newtheorem{theorem}{Theorem}[section]

\newtheorem{thm}[theorem]{Theorem}
\newtheorem{prop}[theorem]{Proposition}

\newtheorem{lemma}[theorem]{Lemma}

\newtheorem{remark}[theorem]{Remark}

\numberwithin{equation}{section}
\numberwithin{equation}{section}
\newtheorem{pro}[theorem]{Proposition}
\newtheorem{hypo}[theorem]{Hypothesis}

\newtheorem{de}[theorem]{Definition}

\newtheorem{rem}[theorem]{Remark}

\def\si{{\sigma}}

\def\RR{\mathbb{R}}

\def\EE{\mathbb{E}}

\def\cF{{\cal F}}

\def\si{{\sigma}}

\def\si{{\sigma}}

\def\om{{\omega}}

\def\dom{{\rm Dom}}

\def\dnW{{W}}

\begin{document}
	
\date{}
\title{Drift parameter estimation for nonlinear stochastic differential equations driven by fractional Brownian motion}

  \author{ Yaozhong Hu \and David Nualart \and Hongjuan Zhou}

 \maketitle
 
 \begin{abstract}
We derive the strong consistency of  the least squares estimator for the drift coefficient  of a fractional stochastic differential system. The drift coefficient  is one-sided dissipative Lipschitz and the  driving  noise  is additive  and fractional with   Hurst parameter $H \in  (\frac{1}{4}, 1)$.    
We assume that continuous observation is possible.  The main tools are
ergodic theorem and Malliavin calculus. As a by-product, we  derive a maximum inequality for  Skorohod integrals, which plays an  
important role to obtain the strong consistency of the least squares estimator.
 \end{abstract}
 
\textbf{Keywords.} Fractional Brownian motion, parameter estimation, nonlinear stochastic differential equation, one-sided dissipative Lipschitz condition, maximum inequality,  moment estimate, H\"older continuity,  strong consistency. 

\vspace{10pt}

\numberwithin{equation}{section}

\section{Introduction and main result}

In this paper, we study a parameter estimation problem for the following stochastic differential equation (SDE)  driven by a fractional Brownian motion (fBm)
  \begin{equation}
   dX_t =   -f(X_t) \theta  dt +   \sigma dB_t \,, \quad t\ge 0 \,,
  \label{nl.sde}
  \end{equation}
where $X_0=x_0 \in \mathbb{R}^m$ is a given initial condition.  The notations appearing in the above equation are explained as follows.  For the diffusion part, $B=(B^1, \dots, B^d)$ is a $d$-dimensional  fBm of Hurst parameter $H \in (0,1)$. The diffusion coefficient $\sigma=(\si_1, \dots, \si_d)$ is an $m\times d$ matrix, with $ \sigma_j$, $j=1, \dots, d
$
being   given vectors in $\RR^m$. For the drift part, the function $f:\RR^m\rightarrow \RR^{m\times l}$  satisfies some regularity and growth  conditions that we shall specify below.  We write $f(x)=(f_1(x), \dots, f_l(x))$, with $f_j(x)$, $j=1, \dots, l$, being vectors in $\RR^m$.
We assume that $\theta= (\theta_1, \dots, \theta_l) \in \mathbb{R}^l$ is an  unknown constant parameter.  
In equation (\ref{nl.sde}) we have used matrix notation, where the vectors are  understood as column vectors. 
With above notations, we may write \eqref{nl.sde}
as
\[
dX_t=-\sum_{j=1}^l \theta_j f_j(X_t)dt+\sum_{j=1}^d \si_jdB_t^j\,.
\]

Our objective is to  estimate the parameter vector $\theta$, from the continuous  observations of the process $X=\{X_t, t\ge 0 \}$  in a finite interval $[0,T]$. 
 We consider a least squares type estimator, which consists of minimizing formally the quantity $\int_0^T |\dot X_t + f(X_t) \theta  |^2dt $, where and in what follows we   use $|\cdot|$ to denote the Euclidean norm of a vector or the Hilbert-Schmidt norm of a matrix.  From this procedure, the least squares estimator (LSE) is given explicitly by 
\begin{equation}
  \hat\theta_T = - \left(\int_0^T (f^{tr}f)(X_t) dt\right)^{-1} \int_0^T   f^{tr}(X_t)  dX_t  \,,
\end{equation}
where $f^{tr}$ denotes the transpose of the matrix $f$.
Substituting \eqref{nl.sde} into the above expression we have 
\begin{equation} \label{theta.est}
  \hat\theta_T =  \theta - \left(\int_0^T  (f^{tr}f)(X_t) dt\right)^{-1}   \int_0^T f^{tr}(X_t) \sigma dB _t \,.
\end{equation}
In the above equation,
the stochastic integral with respect to the fBm is understood as a divergence integral (or Skorohod integral). See Section  2 for  its  definition.

In order to  state  the main result of the paper, we introduce the following hypothesis. 

\begin{hypo} \label{f.cond12}
The functions $f_j$, $1\le j\le m$ are continuously differentiable and there is a positive constant $L_1$
such that 	the Jacobian matrices  $\nabla f_j(x) \in \mathbb{R}^{m \times m}$ satisfy $\sum\limits_{j=1}^l \theta_j \nabla f_j (x) \geq L_1 I_m$ for all $x\in \R^m$, where $I_m$ is the $m \times m$ identity matrix.  
\end{hypo}
In the above hypothesis  and in what follows we use the notation $A\ge B$ to denote the fact that $A-B$ is a non-negative definite matrix.

We denote by $\mathcal{C}^1_p(\RR^m)$ the class of functions $g\in \mathcal{C}^1(\RR^m) $ such that
 there are two positive constants $L_2$ and $\gamma$  with
	\begin{equation}  \label{1.4}
		|g(x)| + |\nabla g(x)| \leq L_2 (1+|x|^\gamma) \,,
	\end{equation}
for  all $x\in \RR^m$. We denote by $\mathcal{C}^2_p(\RR^m)$ the class of functions $g\in \mathcal{C}^2(\RR^m) $ such that
 there are two positive constants $L_2$ and $\gamma$  with
\begin{equation}   \label{1.5}
		|g(x)| + |\nabla g(x)| +  |\mathds{H}(g)(x) | \leq L_2 (1+|x|^\gamma) \,,
	\end{equation}
for  all $x\in \RR^m$,
where $\mathds{H}(g) =\left(\frac{\partial ^2 g}{\partial x_i\partial x_j}    \right)_{1\le i,j\le m}$  denotes  Hessian matrix of  $g$.

It is easy to see that under Hypothesis \ref{f.cond12}, $f$ satisfies the one-sided dissipative Lipschitz condition:
\begin{equation} 
\langle x - y,   (f(x) - f(y)) \theta \rangle \geq   L_{1} |x-y|^2 \,,\quad \forall \ x, y\in \RR^m\,.   \label{1.6}
\end{equation}
According to  the papers \cite{gkn, mh, nt} and the references therein, under  Hypothesis \ref{f.cond12} and  assuming  $f_{ij}\in\mathcal{C}^1_p(\RR^m)$,
for all $1\le i \le m$, $1\le j \le l$,
the SDE \eqref{nl.sde} admits a unique solution $X_t$ in $\mathcal{C}^{\alpha}(\mathbb{R}_+; \mathbb{R}^m)$ for all $\alpha < H$.  Now we state the main result of this paper. 
 
\begin{theorem} \label{thm.cons}
Assume Hypothesis  \ref{f.cond12} and  that the components of $f$ belong to $\mathcal{C}^1_p(\RR^m)$ when $H \in [\frac{1}{2}, 1)$, and  they belong to
$\mathcal{C}^2_p(\RR^m)$  when $H \in (\frac{1}{4}, \frac{1}{2})$. 
Suppose that   $ \mathbb{P} \left(  \det (f^{tr}f)(\overline{X}) >0 \right) >0$, where $\overline{X}$ is the  random variable  appearing in Theorem  \ref{ergodic}.
 Then the least squares estimator  $\hat \theta_T$ of the parameter $\theta$ is strongly consistent in the sense that $
 \displaystyle \lim_{T\rightarrow\infty} |\hat \theta_T-\theta|=0$  almost surely.
\end{theorem}

\begin{remark}
 Condition $ \mathbb{P} \left(  \det (f^{tr}f)(\overline{X}) >0 \right) >0$ means that $\nu (\det (f^{tr}f) >0)>0$, where $\nu $ is the invariant measure of the SDE (\ref{nl.sde}).  A sufficient condition for this  to hold is  $\det(f^{tr}f) (x) >0$ for all $x\in \R^m$. 
 \end{remark}

\begin{remark}
When $f(x)=x$ is linear, this inference problem of $\theta$ has been extensively studied in the literature and various kinds of estimation methods are proposed. We refer interested readers to \cite{hn,hnz} and  the references therein. 

For a general nonlinear case, let us first mention the paper \cite{tv}
in which the maximum likelihood estimator is analyzed.
The paper \cite{nt} is more related to our work, where Neuenkirch and Tindel studied the discrete observation case  and proved the strong consistency of the following estimator 
\[
\bar\theta_n = {\rm argmin}_{\theta} \left|\frac{1}{n\alpha_n^2}\sum_{k=0}^{n-1}\left(|X_{t_{k+1}} - X_{t_k} - f(X_{t_k}; \theta) \alpha_n|^2 - \sum_{j=1}^d |\sigma_j|^2 \alpha_n^{2H}\right) \right|
\]
when $H>\frac{1}{2}$, where $\alpha_n = t_k - t_{k-1}$ satisfies that $\alpha_n n^\alpha$ converges to a constant as $n \to \infty$ for some small $\alpha>0$. 
Their approach relies on Young's inequality from the rough path theory to handle Skorohod integrals, which cannot be applied for the case $H \in (0, \frac{1}{2}]$.
\end{remark} 

We will give the proof of our main theorem in Section \ref{proof}.  The proof relies on a maximum inequality for  Skorohod integrals  which will be presented   in Section \ref{maxineq}.  The main tools we use are Malliavin calculus and ergodic theorem, which will be recalled in Section \ref{s.preliminary}.  

\section{Preliminaries}\label{s.preliminary}

First, let us recall an ergodic theorem for the solution to equation (\ref{nl.sde}) that is crucial for  our
arguments. 
The $d$-dimensional fBm  $B = \{(B^1_t, \dots, B^d_t), t\ge 0 \}$ with Hurst parameter $H \in (0,1)$, is a zero mean Gaussian process   whose components are independent and have the  covariance function 
\begin{equation}
 \mathbb{E}(B_t^i B_s^i)= R_H(t,s):=\frac{1}{2}(|t|^{2H}+|s|^{2H}-|t-s|^{2H}),
 \label{1eq1}
\end{equation}
for $i=1,\dots, d$.
The probability space $(\Omega, \mathcal{F}, \mathbb{P})$ we are taking  is the canonical probability space  of the  fractional Brownian motion. Namely,  $\Omega = C_0(\mathbb{R_+}; \mathbb{R}^d)$ is the set of continuous functions from $\RR_+$
to $\RR^d$ equipped with the uniform  topology on any compact 
interval;  $\mathcal{F}$ is the Borel $\sigma$-algebra, and $\mathbb{P}$ is the   probability measure on $(\Omega, \cF)$ 
such that the coordinate process 
$  B_t(\om)=\om(t)$  is a fractional Brownian motion with Hurst parameter $H \in (0,1)$.

We define the shift operators $\mu_t: \Omega \to \Omega$ as $$\mu_t \omega(\cdot) = \omega(\cdot + t) - \omega(t), \ t \in \mathbb{R}, \omega \in \Omega \,.$$ 
The probability measure $\mathbb{P}$ is invariant with respect to the shift operators $\mu_t$. 
The ergodic property of the SDE \eqref{nl.sde} is summarized in the following theorem (see \cite{mh,nt}).

\begin{theorem} \label{ergodic}
Assume the drift function $f$  satisfies Hypothesis \ref{f.cond12} and  its components belong to $\mathcal{C}^1_p(\RR^m)$.  Then, the following results hold:
\begin{enumerate}
 \item [(i)] There  exists a random variable $\overline X: \Omega \to \mathbb{R}^m$ with   $\mathbb{E}|\overline X|^p < \infty$ for all $p \geq 1$ such that 
\begin{equation}
    \lim_{t \to \infty} |X_t(\omega) - \overline X(\mu_t \omega)| = 0  
\end{equation}
 for $\mathbb{P}$-almost all $\omega \in \Omega$.  

 \item [(ii)]For any function $g \in \mathcal{C}^1_p(\mathbb{R}^m)$,  we have 
\begin{equation}
 	\lim_{T \to \infty} \frac{1}{T} \int_0^T g(X_t) dt = \mathbb{E}[ g(\overline X) ] \qquad \text{$\rm P$-a.s}.
\end{equation}

\end{enumerate}

\end{theorem}

Next  we recall some background material on  the Malliavin calculus for the  fBm $B$. 
 Let $\mathcal{E}^d$ denote the set of  $\mathbb{R}^d$-valued step functions on $[0,\infty)$ with compact support. The Hilbert space $\mathfrak{H}^d$ is defined as the closure of $\mathcal{E}^d$ endowed with the inner product
$$
\langle ( \mathbb{1}_{[0,s_1]} , \dots , \mathbb{1}_{[0,s_d]}),  ( \mathbb{1}_{[0,t_1]} , \dots , \mathbb{1}_{[0,t_d]}) \rangle_{\mathfrak{H}^d}= \mathbb{E} \left[   \left(\sum_{j=1}^d B^j_{s_j}  \right) \left(\sum_{j=1}^d B^j_{t_j} \right) \right]= \sum\limits_{i=1}^d R_H(s_i, t_i) \,.
$$
 Then the mapping $(\mathbb{1}_{[0,t_1]},  \dots , \mathbb{1}_{[0, t_d]}) \mapsto \sum_{j=1}^d B^j_{s_j} $  can be extended to a linear isometry between $\mathfrak{H}^d$ and the Gaussian space $\mathcal{H}_1$ spanned by $B$. We denote this isometry by $\varphi \in \mathfrak{H}^d \mapsto B(\varphi)$. 
 For $d=1$, we simply write $\mathcal{E}=\mathcal{E}^1$ and $\mathfrak{H}=\mathfrak{H}^1$.

When $H=\frac{1}{2}$, $B$ is just  a $d$-dimensional Brownian motion and $\mathfrak{H}^d= L^2( [0,\infty); \R^d)$. When $H \in (\frac{1}{2}, 1)$, let $\mathfrak{|H|}^{d}$ be the linear space of  $\mathbb{R}^d$-valued measurable functions $\varphi$ on $[0,\infty)$ such that
\[
  \|\varphi\|_{|\mathfrak{H}|^{ d}}^2 = \alpha_H \sum_{j=1}^d  \int_{[0,\infty)^{2}} |\varphi^j_{r}||\varphi^j_{s}| |r-s|^{2H-2}  drds  < \infty \,,
\]
where $\alpha_H = H(2H-1)$. Then $|\mathfrak{H}|^{ d}$ is a Banach space with the norm $\|\cdot\|_{|\mathfrak{H}|^ d}$ and $\mathcal{E}^{d}$ is dense in $|\mathfrak{H}|^{ d}$.  Furthermore, for any $\varphi \in L^{\frac{1}{H}}([0,\infty); \mathbb{R}^d)$, we have
\begin{equation} \label{lh.norm}
  \|\varphi\|_{|\mathfrak{H}|^{ d}} \leq b_{H,d} \|\varphi\|_{L^{\frac{1}{H}}([0,\infty);\mathbb{R}^d)} \,,
\end{equation}
for some constant $b_{H,d} > 0$ (See \cite{nu}). Thus, we have  continuous embeddings $   L^{\frac{1}{H}}([0,\infty);\mathbb{R}^d) \subset |\mathfrak{H}|^{d} \subset \mathfrak{H}^{ d}$ for $H > \frac{1}{2}$.

When $H \in (0,\frac{1}{2})$, the covariance of the fBm $B^j$ can be expressed as
\[
 R_H(t,s) = \int_0^{s \wedge t} K_H(s,u)K_H(t,u) du \,,
\]
where $K_H(t,s)$ is a square integrable kernel defined as
\[
  K_H(t,s) = d_H \left(\left(\frac{t}{s}\right)^{H-\frac{1}{2}} (t-s)^{H-\frac{1}{2}} - (H-\frac{1}{2}) s^{\frac{1}{2}-H} \int_s^t v^{H-\frac{3}{2}} (v-s)^{H-\frac{1}{2}} dv\right) \,,
\]
for $0 <s<t$, with $d_H$ being a constant depending on $H$ (see \cite{nu}). The kernel $K_H$ satisfies the following estimates
\begin{equation}\label{kh.est1}
	|K_H(t,s)| \leq c_H \left((t-s)^{H-\frac{1}{2}} + s^{H-\frac{1}{2}}\right) \,,
\end{equation}
and
\begin{equation}\label{kh.est2}
	\left|\frac{\partial K_H}{\partial t} (t,s)\right| \leq c'_H (t-s)^{H-\frac{3}{2}} \,,
\end{equation}
for all $s<t$ and for some constants $c_H, c'_H$. 
Now we define a  linear operator $K_{H}$ from $\mathcal{E}^d$ to $L^2([0,\infty); \RR^d)$ as
\begin{equation}\label{kstar}
  K_{H}(\phi)(s) = \left( K_H(T,s)\phi(s) + \int_s^T (\phi(t) - \phi(s)) \frac{\partial K_H}{\partial t} (t,s) dt  \right) \mathbb{1}_{[0,T]} (s)\,,
\end{equation}
where the support of $\phi$ is included in $[0,T]$. One can show that this definition does not depend on $T$.
Then the operator $K_{H}$ can be extended to an isometry between the Hilbert space $\mathfrak{H}^d$ and $L^2([0,\infty);\RR^d)$ (see \cite{nu}), and if
$\phi \in \mathfrak{H}^d$ has support in $[0,T]$, then  (\ref{kstar}) holds.
 For $\phi  \in  \mathfrak{H}^d$ with support in $[0,T]$, we  define
\begin{eqnarray*}
  \|\phi\|_{K^{d}_T} ^2 &: = & \int_0^T |\phi(t)|^2 \left((T-t)^{2H-1} + t^{2H-1}\right) dt   +   \int_0^T \left(\int_s^T |\phi(t)-\phi(s)| (t-s)^{H-\frac{3}{2}}dt \right)^2 ds \,.
\end{eqnarray*}
By the estimates \eqref{kh.est1} and \eqref{kh.est2}, there exists a constant $C$ depending on $H$ such that for any $\phi \in \mathfrak{H}^{d }$ with  support in $[0,T]$,
\begin{equation}\label{hnorm.est}
  \|\phi\|^2_{\mathfrak{H}^{d }} = \|K_H(\phi)\|^2_{L^2([0,\infty);\mathbb{R}^d)} \leq C \|\phi\|^2_{K^{d}_T} \,.	
\end{equation}

Next, we introduce the derivative operator and its adjoint, the divergence. Consider a smooth and cylindrical random variable of the form $F= f(B_{t_1}, \dots, B_{t_n})$, where $f \in C_b^{\infty}(\mathbb{R}^{d \times n})$ ($f$ and its partial derivatives are all bounded). We define its Malliavin derivative as the $\mathfrak{H}^d$-valued random variable given by $DF = (D^1F, \dots, D^d F)$ whose $j$th component is  given by 
$$
 D^j_s F = \sum_{i=1}^n \frac{\partial f}{\partial x_i^j} (B_{t_1},\dots, B_{t_n})\mathbb{1}_{[0, t_j]}(s).
 $$
By iteration, one can define  higher order derivatives $D^{j_1, \dots, j_i}F$ that take values on $(\mathfrak{H}^d)^{\otimes i}$. For any natural number $p$ and any real number $ q \geq 1$, we define the Sobolev space $\mathbb{D}^{p,q}$  as the closure of the space of smooth and cylindrical random variables with respect to the norm $\|\cdot\|_{p,q}$ given by
$$
\|F\|^q_{p,q} = \mathbb{E}(|F|^q) +  \sum_{i=1}^p \mathbb{E}  \left[ \left(\sum_{j_1, \dots, j_i = 1}^d \|D^{j_1, \ldots, j_i} F\|^2_{(\mathfrak{H}^d)^{\otimes i}}\right)^{\frac{q}{2}} \right].
$$
Similarly, if $\W$ is a general Hilbert space, we can define the Sobolev space of $\W$-valued random variables $\mathbb{D}^{p,q}(\W)$.

For $j=1, \ldots, d$, the adjoint of the Malliavin derivative operator $D^j$, denoted as $\delta^j$, is called the divergence operator or Skorohod integral (see \cite{nu}). A random element $u$ belongs to the domain of $\delta^j$, denoted as $\dom(\delta^j)$, if there exists a positive constant $c_u$ depending only on $u$ such that
$$
|\mathbb{E}( \langle D^j F, u \rangle_{\mathfrak{H}}) | \leq c_u \|F\|_{L^2(\Omega)}
$$ 
for any $F \in \mathbb{D}^{1,2}$.  If $u \in \dom(\delta^j)$, then the random variable $\delta^j(u)$ is defined by the duality relationship
\[
  \mathbb{E} \left(F \delta^j(u)\right) = \mathbb{E}(\langle D^j F, u \rangle_{\mathfrak{H}}) \, ,
\] for any $F \in \mathbb{D}^{1,2}$.  In a similar way, we can define the divergence operator on $\mathfrak{H}^d$ and  we have
 $\delta(u) = \sum_{j=1}^d \delta^j (u_j)$ for $u=(u_1, \dots, u_d) \in \cap_{j=1}^d \dom(\delta^j)$. We make use of the notation $\delta(u)=\int_0^\infty u_t dB_t$ and call $\delta(u)$ the divergence integral of $u$ with respect to the fBm $B$.

For $p>1$, as a consequence of Meyer's inequality, the divergence operator $\delta$ is continuous from $\mathbb{D}^{1,p}(\mathfrak{H}^d)$ into $L^p(\Omega)$, which means
\begin{equation} \label{div.pm}
  \mathbb{E}(|\delta(u)|^p) \leq C_p \left(\mathbb{E}(\| u\|^p_{\mathfrak{H}^d}) + \mathbb{E}(\|Du\|^p_{\mathfrak{H}^{d } \otimes \mathfrak{H}^{d }})\right) \,,
\end{equation}
for some constant $C_p$ depending on $p$.

\section{Moment estimates and maximal inequality for divergence integrals with respect to fBm } \label{maxineq}

When $H>\frac{1}{2}$, thanks to \eqref{lh.norm} and \eqref{div.pm}, the following lemma provides a useful  estimate for the $p$-norm  of the divergence integral with respect to fBm.
\begin{lemma} \label{pnorm.g}
Let $H \in (\frac{1}{2}, 1)$ and let $u$ be an element of $\mathbb{D}^{1,p}(\mathfrak{H}^d)$, $p>1$. Then $u$ belongs to the domain of the divergence operator $\delta$ in $L^p(\Omega)$. Moreover, we have
\[
  \mathbb{E}(|\delta(u)|^p) \leq C_{p,H} \left( \| \mathbb{E}(u) \|^p_{L^{1/H}([0,\infty); \mathbb{R}^d)} + \mathbb{E}\left(\|Du\|^p_{L^{1/H}([0,\infty)^2; \mathbb{R}^{d\times d} }\right) \right) \,.
\]
\end{lemma}
Now we consider the case of $H \in (0,\frac{1}{2})$. First we will derive an estimate for the $p$-norm of  $\| u \mathbb{1}_{[a,b]}\|_{\mathfrak{H}\otimes \W}$, where $u$ is a stochastic process with values in a Hilbert space $\W$.

Consider the functions $L^t$ and $L^{t,s}$ defined  for $0< s<t <b$ by
\begin{equation} \label{lt}
	L^t(\lambda_0, \lambda_1) :=   (b-t)^{\lambda_0} t^{\lambda_1} \,,
\end{equation}
\begin{equation} \label{lts}
	L^{t,s}(\lambda_2, \lambda_3, \lambda_4) :=  (b-t)^{\lambda_2} (t-s)^{\lambda_3}  s^{\lambda_4} \,.
\end{equation}
 where the $\lambda_i$'s are parameters. We  denote by $C$ a generic constant that depends only on the coefficients of the SDE \eqref{nl.sde}, the Hurst parameter $H$ and the  parameters introduced along the paper. 

\begin{pro} \label{u.pmom}
Let $p \geq 2$ and  $H \in (0,\frac{1}{2})$. Fix $b\geq 0$. Let $\W$ be a Hilbert space and  consider a $\W$-valued stochastic   process $u=\{u_t, t \ge 0\}$ satisfying the following conditions:
	\begin{itemize}
		\item [(i)] $\|u_t\|_{L^p(\Omega; \W )} \leq  K_1 L^t(\lambda_0, \lambda_1)$, for all $t \ge 0$; 
		\item [(ii)] $\|u_t - u_s\|_{L^p(\Omega;  \W )} \leq  K_2 L^{t,s}(\lambda_2, \lambda_3, \lambda_4)$, for all $s < t\le b$,
	\end{itemize}
where the parameters $\lambda_i$   satisfy $\lambda_0  >-H$, $\lambda_1  ,\lambda_4 \ge 0$, $\lambda_2  >- \frac 12$, and $\lambda_3 > \frac{1}{2}-H$. Then for all $0 \leq a \leq b$,
   \begin{eqnarray}
     \mathbb{E} (\|u\mathbb{1}_{[a,b]}\|^p_{\mathfrak{H} \otimes \W  }) & \leq &  CK_2^p  b^{p\lambda_4} (b-a)^{pH + p\lambda_2 + p\lambda_3} + \ CK_1 ^p  b^{p\lambda_1}(b-a)^{pH+p\lambda_0 } \,.   \label{u.pnorm} 
   \end{eqnarray} 
\end{pro}

\begin{proof}
To simplify we assume $\W=\R$.
 Using the isometry   of the operator $K_{H}$, we  can write
 \[
   \mathbb{E} (\|u\mathbb{1}_{[a,b]}\|^p_{\mathfrak{H} }) =  \mathbb{E} \left(\|K_{H}(u\mathbb{1}_{[a,b]})\|^p_{L^2([0,b])}\right) \,.
 \]
We decompose the integral appearing in \eqref{kstar} into sum of
three  terms 
according to the cases where   one of $s, t$ is in the interval $(a,b)$ or both.  In this way, we obtain
 \begin{eqnarray*}
   K_{H}(u\mathbb{1}_{[a,b]}) & = & K_H(b,s) u_s \mathbb{1}_{[a,b]} (s) + \left( \int_s^b (u_t - u_s) \frac{\partial K_H}{\partial t} (t,s) dt \right) \mathbb{1}_{[a,b]} (s) \\
   & &  + \left(\int_a^b u_t \frac{\partial K_H}{\partial t} (t,s) dt\right) \mathbb{1}_{[0,a]} (s)  \\
   &=:& I_1 + I_2 + I_3 \,.
 \end{eqnarray*}
Thus,
   \begin{equation}\label{est}
     \mathbb{E} (\|u\mathbb{1}_{[a,b]}\|^p_{\mathfrak{H} }) \leq C \sum_{i=1}^3 A_i \,,
   \end{equation}
where $A_i=\mathbb{E}\left(\|I_i\|^p_{L^2([0,b])}\right)$. Now we estimate each term  $A_i$ in \eqref{est}. For $A_1$, applying Minkowski inequality and condition $(i)$,
we obtain  
\begin{eqnarray*}
	A_1 & \leq & C \left(\int_a^b \left((b-s)^{2H-1} + s^{2H-1}\right)\|u_s\|^2_{L^p(\Omega)} ds\right)^{\frac{p}{2}} \\
	& \leq & CK_1^p \left(\int_a^b ((b-s)^{2H-1} + s^{2H-1})(b-s)^{2\lambda_0}  s^{2\lambda_1} ds\right)^{\frac{p}{2}} \\
	& \leq & CK_1^p \left(\int_a^b \left((b-s)^{2H-1} + (s-a)^{2H-1}\right) (b-s)^{2\lambda_0} s^{2\lambda_1} ds\right)^{\frac{p}{2}} \\
	& = & CK_1  ^p b^{p\lambda_1} (b-a)^{pH+p\lambda_0} \,.
\end{eqnarray*}
For the term $A_3$, applying again Minkowski inequality and condition $(i)$, we can write
\begin{eqnarray*}
	A_3 & \leq & C \left(\int_0^a \left(\int_a^b \|u_t\|_{L^p(\Omega)} (t-s)^{H-\frac{3}{2}}dt\right)^2 ds\right)^{\frac{p}{2}} \\
	& \leq & CK_1 ^p \left(\int_0^a \left(\int_a^b (b-t)^{\lambda_0} t ^{\lambda_1} (t-s)^{H-\frac{3}{2}} dt\right)^2 ds\right)^{\frac{p}{2}}\,.
\end{eqnarray*}
Denote  $g(t)=(b-t)^{\lambda_0} t^{\lambda_1}$ which is positive. Then
\[
  A_3  \leq  C K_1 ^p   \left(\int_{[a, b]^2} g(t_1) g(t_2) dt_1 dt_2\int_0^a (t_1-s)^{H-\frac32}(t_2-s)^{H-\frac32}ds\right)^{\frac{p}{2}} \,.
\]
Now 
\begin{eqnarray*} 
\int_0^a (t_1-s)^{H-\frac32}(t_2-s)^{H-\frac32}ds
\le \int_0^a (t_1-a)^{H-\frac32}(t_2-s)^{H-\frac32}ds
\le C (t_1-a)^{H-\frac32} (t_2-a)^{H-\frac{1}{2}}
\,.
\end{eqnarray*}
In the same way we have 
\begin{eqnarray*} 
\int_0^a (t_1-s)^{H-\frac32}(t_2-s)^{H-\frac32}ds
\le C (t_2-a)^{H-\frac32} (t_1-a)^{H-\frac{1}{2}}
\,.
\end{eqnarray*}
Using the fact that if $u\le a_1$ and $u\le a_2$,  then $u\le \sqrt{a_1a_2}$, we see that
\[
\int_0^a (t_1-s)^{H-\frac32}(t_2-s)^{H-\frac32}ds
\le (t_1-a)^{H-1} (t_2-a)^{H-1}\,.
\]
Therefore, we have
\begin{eqnarray*}
	A_3
	& \leq & CK_1^p     \left(\int_a^b (b-t)^{\lambda_0} (t-a)^{ H - 1  } t^{\lambda_1} dt   \right)^{p} 
	\leq CK_1 ^p b^{p\lambda_1} (b-a)^{pH+p\lambda_0   }   \,.
\end{eqnarray*}
 For $A_2$, applying Minkowski inequality and condition $(ii)$, yields
\begin{eqnarray*}
    A_2 & \leq & C \left(\int_a^b \left(\int_s^b \|u_t-u_s\|_{L^p(\Omega)} (t-s)^{H-\frac{3}{2}} dt\right)^2 ds\right)^{\frac{p}{2}} \\
	& \leq & CK_2 ^p  \left(\int_a^b \left(\int_s^b (b-t)^{\lambda_2} (t-s)^{\lambda_3}  s^{\lambda_4} (t-s)^{H-\frac{3}{2}} dt\right)^2 ds\right)^{\frac{p}{2}} \\
	& \leq & CK_2^p   \left(\int_a^b (b-s)^{2\lambda_2 + 2\lambda_3 + 2H - 1} s^{2\lambda_4} ds\right)^{\frac{p}{2}}\\
	& = & CK_2^p b^{p\lambda_4} (b-a)^{pH + p\lambda_2 + p\lambda_3} \,.
\end{eqnarray*}
This completes the proof.
\end{proof}

Suppose now that  $u$ is a $d$-dimensional stochastic process.
We will make use of the notation $\|u\|_{p,a,b} := \sup_{a \leq t \leq b} \|u_t\|_{L^p(\Omega; \mathbb{R}^d)}$. Consider the following regularity conditions on $u$:  

\begin{hypo}\label{hypo.u} Assume that there are constants 
 $K > 0$,   $\beta > \frac{1}{2} - H$ and $\lambda \in (0,H]  $, such that
the $\mathbb{R}^d$-valued process $u=\{u_t , t \ge 0 \}$ and its derivative $\{Du_t, t  \ge 0\}$ 
satisfy the following conditions:
	\begin{itemize}
		\item [(i)] $\|u\|_{p,0,\infty}= \sup_{t \geq 0} \|u_t\|_{L^p(\Omega; \mathbb{R}^d)} < \infty$,
		\item [(ii)] $\|u_t - u_s\|_{L^p(\Omega; \mathbb{R}^d)} \leq K (t-s)^{\beta}$,
		\item [(iii)] $\|D u_t\|_{L^p(\Omega;   \HH^d \otimes \R^d)} \leq K  t^\lambda$,
		\item [(iv)] $\|Du_t - Du_s\|_{L^p(\Omega; \HH^d \otimes \R^d)} \leq K (t-s)^{\beta}  s^\lambda$,
	\end{itemize}
for all $0 \leq  s < t $.	 
\end{hypo}

  As an application of \eqref{div.pm} and Proposition \ref{u.pmom}, 
we give the following estimate for the $p$-th moment of the divergence integral $\delta(u\mathbb{1}_{[0,T]})$.
\begin{prop} \label{div.pmom}
Let $H \in (0,\frac{1}{2})$ and $p \geq 2$.  Assume that the $\mathbb{R}^d$-valued stochastic process $\{u_t , t \ge 0\}$ satisfies Hypothesis \ref{hypo.u}. Then for any $T>0$,  the divergence integral $\delta(u\mathbb{1}_{[0,T]})$ is in $L^p(\Omega)$, and 
  \[
    \mathbb{E}( |\delta(u\mathbb{1}_{[0,T]})|^p) \leq C T^{pH} (1+ T^{p\lambda} ) (1+ T^{p \beta} ) \,,
  \]
where the constant $C$   is independent of $T$.
\end{prop}
\begin{proof}
	We will use  inequality \eqref{div.pm} to prove the proposition and it suffices to compute the right-hand side of \eqref{div.pm}.   Applying  Proposition \ref{u.pmom} to  $\W =\mathbb{R}^d$, $\lambda_3 = \beta$ and $\lambda_i=0, i \neq 3$, we obtain
	\[
	\mathbb{E}( \|u\mathbb{1}_{[0,T]}\|^p_{\mathfrak{H}^d}) \leq C\left(\|u\|^p_{p,0,\infty} T^{pH} + K ^pT^{p\beta + pH} \right) \,.
	\]
 
To compute the $p$-th moment of the derivative of $u$, we use the functions $L^t$ and $L^{t,s}$  introduced in (\ref{lt}) and (\ref{lts}), respectively, to write the conditions (iii) and (iv) of Hypothesis \ref{hypo.u} as
\[
  \|D u_t\|_{L^p(\Omega; \HH^d \otimes  \R^d)} \leq K L^t(0, \lambda) \,,
\]
and
\[
  \|Du_t - Du_s\|_{L^p(\Omega; \HH^d \otimes \R^d)} \leq K L^{t,s}(0, \beta,  \lambda) \,.
\]
Then we use Proposition \ref{u.pmom} for $\W =\HH ^d \otimes \R^d$ and take into account the isomorphism $\mathfrak{H} \otimes (\mathfrak{H}^d \otimes \mathbb{R}^d) \cong \mathfrak{H}^d \otimes \mathfrak{H}^d$ to obtain
\begin{equation*}
\mathbb{E} (\|Du\mathbb{1}_{[0,T]}\|^p_{\HH ^d \otimes \HH^d})  \leq C K^p  T^{pH+p\lambda} (1+ T^{p \beta} )  \,.
\end{equation*}
This completes the proof of the proposition. 
\end{proof}
When $H \neq \frac{1}{2}$, the divergence integral $\left\{ \int_0^t u_s dB_s\,,  t\ge 0\right\}$ is not a martingale, so we cannot apply Burkholder inequality to bound the maximum of the integral. However,
if the process $u$ satisfies some regularity conditions in Hypothesis \ref{hypo.u}, we can use a factorization method to estimate the maximum, as it has been done in
\cite{AN}. This result is given in the following theorem.

\begin{thm} \label{div.maxineq}  
 Let $\{u_t, t  \ge 0 \}$ be an $\mathbb{R}^d$-valued stochastic process. For the divergence integral $\int_0^t u_s dB_s$, $t  \ge 0$, we have the following statements:
\begin{enumerate}
\item Let $H \in (\frac{1}{4},\frac{1}{2})$ and $p> \frac{1}{H}$. Assume that the stochastic process $u$ satisfies Hypothesis \ref{hypo.u}. Then the divergence integral $\int_0^t u_s dB_s$ is in $L^p(\Omega)$ for all $t\ge 0$ and for any $0\le a<b$ we have the estimate
	\[
	 \mathbb{E}\left(\sup_{t \in [a,b]} \left| \int_a^t u_s dB_s \right|^p\right) \leq C (b-a)^{pH} (1+(b-a)^{p\beta}  ) (1+ b^{p\lambda}) \,,
	\]
 where $C$  is a generic constant that does not depend on $a,b$.
\item Let $H \in (\frac{1}{2}, 1)$ and $\frac{1}{p} + \frac{1}{q} = H $  with $p>q$. Suppose that  for all $T>0$
 \begin{itemize}
 \item[(i)] $\int_0^T \mathbb{E}(|u_s|^p )ds < \infty$,
 \item[(ii)] $\int_0^T \int_0^s \mathbb{E}(|D_t u_s|^p )dt ds < \infty$.
 \end{itemize}
Then the divergence integral $\int_0^t u_s dB_s$ is in $L^p(\Omega)$ for all $t\ge 0$ and for any interval $[a,b]$, we have
\begin{eqnarray*}
 & & \mathbb{E}\left(\sup_{t \in [a,b]} \left| \int_a^t u_s dB_s \right|^p\right) \leq C \left( (b-a)^{\frac{p}{q}} \int_{a}^b \mathbb{E}(|u_s|^p )ds +  (b-a)^{\frac{2p}{q}} \int_{a}^b \int_a^s \mathbb{E}(|D_t u_s|^p )dt ds \right) \,,
\end{eqnarray*}
where the constant $C$ does not depend on $a, b$.
\end{enumerate}
\end{thm}

\begin{proof}
We may assume that $u$ is a smooth function. The general case follows from a limiting  argument.   
 We will use the elementary integral $\int_s^t (t-r)^{\alpha-1}(r-s)^{-\alpha} dr = \frac{\pi}{\sin(\alpha \pi)}$  for any $\alpha\in (0, 1)$, and a stochastic  Fubini's theorem.  For any $\alpha \in (\frac{1}{p}, 1)$,  we have 
\begin{eqnarray}
&&  \mathbb{E} \left(\sup_{t \in [a,b]} \left|\int_a^t u_s dB_s\right|^p\right) \nonumber \\
&& \quad =  \left(\frac{\sin(\alpha \pi)}{\pi}\right)^p \mathbb{E} \left(\sup_{t \in [a,b]} \left|\int_a^t \left (\int_s^t (t-r)^{\alpha-1}(r-s)^{-\alpha} dr \right) u_s dB_s \right|^p  \right)\nonumber \\
 & &  \quad  \left(\frac{\sin(\alpha \pi)}{\pi}\right)^p \mathbb{E} 
 \left( \sup_{ t\in [a,b]} \left|\int_a^t  \left(\int_a^r (r-s)^{-\alpha} u_s dB_s\right) (t-r)^{\alpha-1} dr \right|^p\right)\nonumber \\
 & &  \quad \leq \left(\frac{\sin(\alpha \pi)}{\pi}\right)^p \mathbb{E} \left(\sup_{t \in [a,b]} \left(\int_a^t \left|\int_a^r (r-s)^{-\alpha} u_s dB_s \right|^p dr\right) \left|\int_a^t (t-r)^{\frac{p(\alpha-1)}{p-1}} dr \right|^{p-1}\right)\nonumber \\
 &  &  \quad \leq C_{\alpha, p} (b-a)^{p \alpha -1} \int_a^b \mathbb{E}(|G_r|^p) dr\,,\label{max.ineq}
\end{eqnarray}
where  
\[
  G_r := \int_a^r (r-s)^{-\alpha} u_s dB_s, \qquad r \in [a,b] \,.
\]

\noindent 
{\it Case $H \in (\frac{1}{2}, 1)$}: Using Lemma \ref{pnorm.g} for $\alpha \in (\frac{1}{p}, \frac{1}{q})$ and $\frac{1}{p} + \frac{1}{q} = H$, we get 
\begin{eqnarray*}
  \mathbb{E}(|G_r|^p) & \leq & C_{p,H} \left( \left( \int_a^r (r-s)^{-\frac{\alpha}{H}} |\mathbb{E}(u_s)|^{\frac{1}{H}} ds\right)^{pH} + \mathbb{E}\left(\int_a^r \int_a^s (r-s)^{-\frac{\alpha}{H}} |D_{\mu} u_s|^{\frac{1}{H}} d\mu ds\right)^{pH}\right) \\
  & \leq & C_{p,H} \left(\int_a^r (r-s)^{-\alpha q} ds\right)^{\frac{p}{q}} \left(\int_a^r |\mathbb{E}(u_s)|^p ds\right)\\
  & & + \ C_{p,H} \left(\int_a^r \int_a^s (r-s)^{-\alpha q} d\mu ds \right)^{\frac{p}{q}} \left(\int_a^r \int_a^s \mathbb{E}(|D_{\mu} u_s|)^p d\mu ds\right) \\
  & \leq & C_{\alpha,p,q,H} \left( (r-a)^{\frac{p}{q} - \alpha p} \int_a^r \mathbb{E}(|u_s|^p) ds +  (r-a)^{\frac{2p}{q} - \alpha p} \int_a^r \int_a^s \mathbb{E}(|D_{\mu} u_s|^p) d\mu ds \right) \,.
\end{eqnarray*}
Therefore,
\[
  \mathbb{E} \left(\sup_{t \in [a,b]} \left|\int_a^t u_s dB_s\right|^p\right)  \leq C \left( (b-a)^{\frac{p}{q}} \int_a^b \mathbb{E}(|u_s|^p) ds +  (b-a)^{\frac{2p}{q}} \int_a^b \int_a^s \mathbb{E}(|D_{\mu} u_s|^p) d\mu ds \right) \,.
\]	

\noindent 
{\it Case $H \in (0, \frac{1}{2})$}:  Denote $\psi(t) = (r-t)^{-\alpha} u_t$ for $t \in [a,r)$. Then by \eqref{div.pm},
\begin{equation}\label{gr.pm}
  \mathbb{E} (|G_r|^p) \leq \mathbb{E} (\|\psi \mathbb{1}_{[a,r)}\|_{\mathfrak{H}^d}^p) + \mathbb{E} (\| D(\psi \mathbb{1}_{[a,r)})\|_{\mathfrak{H}^d \otimes \HH^d } ^p) \,.
\end{equation}
We will estimate the above two items on the right-hand side one by one. For $a \leq s < t < r$,
\begin{eqnarray*}
	|\psi(t) - \psi(s)| & = & |(r-t)^{-\alpha} (u_t - u_s) + \left((r-t)^{-\alpha} - (r-s)^{-\alpha}\right) u_s| \\
	& \leq & (r-t)^{-\alpha} |u_t - u_s| + (r-t)^{-2\alpha} (t-s)^\alpha |u_s| \,,
\end{eqnarray*}
where we have used the inequality $1-(r-t)^\alpha (r-s)^{-\alpha} \leq (r-s)^{-\alpha} (t-s)^\alpha$. Thus, using Hypothesis 3.3 (ii), we can write
\begin{eqnarray}
  \|\psi(t) - \psi(s)\|_{L^p(\Omega;\R^d)} \nonumber & \leq& (r-t)^{-\alpha} \|u_t - u_s\|_{L^p(\Omega;\R^d)} + (r-t)^{-2\alpha} (t-s)^\alpha \|u_s\|_{L^p(\Omega;\R^d)} \nonumber\\
  & \leq&  K(r-t)^{-\alpha} (t-s)^\beta + \|u\|_{p,a,b} (r-t)^{-2\alpha}(t-s)^\alpha,
\end{eqnarray}
and \begin{eqnarray}
   \|\psi(t)\|_{L^p(\Omega;\R^d)} &= & (r-t)^{-\alpha} \|u_t \mathbb{1}_{[a,r)}\|_{L^p(\Omega;\R^d)} 
    \leq (r-t)^{-\alpha} \|u\|_{p,a,b}  \,, \label{psi}
\end{eqnarray}
This means that $\psi$ satisfies the assumptions of  Proposition \ref{u.pmom} with $\dnW =\mathbb{R}^d $ with the   functions
$L^t(-\alpha ,0)$ and $L^{t,s}( -\alpha, \beta,0)+ L^{t,s} (-2\alpha, \alpha,0)$ if we  choose $\alpha \in (\max(\frac{1}{p}, \frac{1}{2}-H),H)$, which requires  $H \in (\frac{1}{4}, \frac{1}{2})$.  In this way, we obtain
\begin{eqnarray}
  \mathbb{E} (\|\psi \mathbb{1}_{[a,r]}\|_{\mathfrak{H}^d}^p) &\leq& C (r-a)^{pH-p\alpha}  (1 + (r-a)^{p\beta} ) \label{psi.hnorm}  \,.
\end{eqnarray}
 Similarly,  using Hypotheses 3.3 (iii) and (iv), we have
\begin{eqnarray}
 &&   \|D\psi(t) - D\psi(s)\|_{L^p(\Omega; \mathfrak{H}^d \otimes \R^d)} \nonumber\\
 & \leq & (r-t)^{-\alpha} \|Du_t - Du_s\|_{L^p(\Omega; \mathfrak{H}^d \otimes \R^d)} + (r-t)^{-2\alpha} (t-s)^\alpha \|Du_s\|_{L^p(\Omega; \mathfrak{H}^d \otimes \R^d)} \nonumber \\
	 & \leq & K(r-t)^{-\alpha} (t-s)^\beta  s^\lambda  + \ K(r-t)^{-2\alpha} (t-s)^\alpha  s^\lambda \label{dpsi.dif}
\end{eqnarray}
and
\begin{eqnarray}
   \|D\psi(t)\|_{L^p(\Omega ; \mathfrak{H}^d \otimes \R^d)} &= & (r-t)^{-\alpha} \|Du_t\|_{L^p(\Omega ; \mathfrak{H}^d \otimes \R^d)} 
    \leq   K(r-t)^{-\alpha}  t^\lambda.  \label{dpsi}
\end{eqnarray}
This means that $D\psi$ satisfies the assumptions of  Proposition \ref{u.pmom} with $\W =\HH^d \otimes \R^d$ with the   functions
$L^t(-\alpha ,\lambda)$ and $L^{t,s}( -\alpha, \beta,\lambda)+ L^{t,s} (-2\alpha, \alpha,\lambda)$.
Using Proposition \ref{u.pmom} for $D\psi$ with $\dnW =\mathfrak{H}^d \otimes \R^d$, we have
\begin{eqnarray}
  \mathbb{E} (\| D(\psi \mathbb{1}_{[a,r]})\|_{ \HH^d \otimes \HH^d}^p) 
   &\leq & C  (r-a)^{pH-p\alpha} ( 1+(r-a)^{p\beta})   b^{p\lambda}.
 \label{dpsi.hnorm}  
\end{eqnarray}
Substituting the bounds of \eqref{psi.hnorm} and \eqref{dpsi.hnorm} into \eqref{gr.pm}, we have 
\begin{eqnarray}
  \mathbb{E}(|G_r|^p) & \leq &
  C (r-a)^{pH-p\alpha} (1+(r-a)^{p\beta}) (1+ b^{p\lambda}) \,.
\end{eqnarray}
Finally, putting this estimate into 
 \eqref{max.ineq}, we complete the proof.
\end{proof}

\section {Proof of the main theorem} \label{proof}

\subsection{Estimates of the solution of SDE}
Before we present the proof of the  main theorem, we  need  some  
auxiliary results.  First, we prove some estimates for  the $p$-th moment of the solution of the SDE  (\ref{nl.sde}).

\begin{prop} \label{p.moment}
Let $H \in (0,1)$ and $p\geq 1$.  Assume the drift function $f$  of the SDE (\ref{nl.sde}) satisfies Hypotheses \ref{f.cond12} and  its components belong to $\mathcal{C}^1_p(\RR^m)$. 
 Let  $X$ be the unique solution to  (\ref{nl.sde}).    Then we have the following statements:
\begin{itemize}
	\item[(1)] There exists a constant $C_p > 0$ such that $ \|X_t\|_{L^p(\Omega;\mathbb{R}^m)} \leq C_p$, and $\|X_t - X_s\|_{L^p(\Omega;\mathbb{R}^m)} \leq C_p |t-s|^{H}$ for all $t \geq s \geq 0$.
	\item[(2)] The Malliavin derivative of the solution $X_t$ satisfies for all $0\le s \leq t $
	\begin{equation}\label{ineq.dext}
	  |D_s X_t| \leq |\sigma | e^{-L_1(t-s)} \,, \ {\rm a.s.}
	\end{equation}
	Moreover, if $v \leq u \leq s \leq t$, we have
	\begin{equation}\label{ineq3.dext}
	   \|D_u X_t - D_v X_t\|_{L^p(\Omega; \mathbb{R}^{m \times d})} \leq C  e^{-L_1(t-u)}(1\wedge   |u-v|)  \,,
	\end{equation}
	\begin{equation}\label{ineq2.dext}
	  \|D_u X_t - D_u X_s\|_{L^p(\Omega; \mathbb{R}^{m \times  d})} \leq Ce^{-L_1(s-u)}(1\wedge |t-s| ) \,,
	\end{equation}
	and \begin{equation}\label{ineq4.dext}
	   \|D_u X_t - D_v X_t - (D_u X_s - D_v X_s)\|_{L^p(\Omega; \mathbb{R}^{m\times  d})} \leq C  e^{-L_1(s-u)} (1\wedge   |u-v|) (1\wedge  |t-s|)   \,,
	\end{equation}
where  $C$ is a generic constant.
\end{itemize}
\end{prop}

\begin{proof}
For  the proof of the first result we refer  to   \cite{gkn}, \cite{mh}, and \cite{nt}.

To show  the second  part of this proposition,
taking the Malliavin derivative for $s \leq t$ on  both sides of  equation \eqref{nl.sde} yields
\begin{equation} \label{de.xt}
  D_s X_t =  - \int_s^t \sum_{j=1}^l \theta_j \nabla f_j(X_r) D_s X_r dr + \sigma  \,,
\end{equation}
where $\sigma = (\sigma_1, \ldots, \sigma_d) \in \mathbb{R}^{m  \times d}$. Denote $Z_t = D_s X_t$ for $t\ge s$. We can write the above equation as the following ordinary differential equation for $t\ge s$:
\[
\begin{cases}
  d Z_t = - \sum_{j=1}^l \theta_j \nabla f_j(X_t) Z_t dt, \\
  Z_s = \sigma.
\end{cases}
\]
Differentiating  $|Z_t|^2$ with respect to $t$, and using (\ref{1.6}), we get
\[
  \frac{d|Z_t|^2}{dt} = 2\langle Z_t,  - \sum_{j=1}^l \theta_j \nabla f_j(X_t) Z_t \rangle \leq - 2 L_1 |Z_t|^2 \,.
\]
By Gronwall's lemma, we obtain
\[
  |Z_t|^2 \leq e^{-2L_1(t-s)} \left| \sigma \right|^2 \,,
\]
and this implies \eqref{ineq.dext}. 

We now  proceed  to the proof of  \eqref{ineq3.dext}. For $v \leq u \leq t$,   equation \eqref{de.xt} implies
\begin{equation}\label{md.dif}
  D_uX_t - D_v X_t = -\int_u^t \sum_{j=1}^l \theta_j \nabla f_j(X_r) (D_u X_r - D_v X_r) dr + \int_v^u \sum_{j=1}^l \theta_j \nabla f_j(X_r) D_vX_r dr \,.
\end{equation}
Repeating the above arguments for $D_uX_t -D_vX_t$, $t\ge u$, we can write
\[
  |D_u X_t - D_v X_t| \leq e^{-L_1(t-u)} \Big|\int_v^u \sum_{j=1}^l \theta_j \nabla f_j(X_r) D_v X_r dr\Big| \,.
\]
Applying Minkowski inequality and \eqref{ineq.dext} to  $D_v X_r$, and then using the fact that the $L^p$-norm of $|\nabla f_j(X_r)| $ is bounded due to
condition (\ref{1.4}), we obtain
\begin{eqnarray*}  
  \|D_u X_t - D_v X_t\|_{L^p(\Omega; \mathbb{R}^{m \times d})} &\leq & e^{-L_1(t-u)} \int_v^u \|\sum_{j=1}^l \theta_j \nabla f_j(X_r) D_v X_r \|_{L^p(\Omega; \mathbb{R}^{m \times  d})} dr \\
  &  \leq & C  e^{-L_1(t-u)} \int_v^u e^{-L_1(r-v)} dr \leq  C  e^{-L_1(t-u)} (1 \wedge |u-v|)  \,.
\end{eqnarray*}
This proves \eqref{ineq3.dext}. 
To prove \eqref{ineq2.dext}, we use  equation \eqref{de.xt} to obtain
\begin{equation*}
  \mathbb{E}(|D_uX_t - D_u X_s|^p )= \mathbb{E} \left( \left|   \int_s^t \sum_{j=1}^l \theta_j \nabla f_j(X_r) D_uX_r dr\right|^p \right) \,.
\end{equation*}
Applying Minkowski inequality and using \eqref{ineq.dext} for $D_uX_r$, and the fact that the $L^p$-norm of $|\nabla f_j(X_r)|$ is bounded, we obtain
\begin{eqnarray*}  
  \|D_uX_t - D_u X_s\|_{L^p(\Omega; \mathbb{R}^{m \times d})} &\leq & \int_s^t \left  \|\sum_{j=1}^l \theta_j \nabla f_j(X_r) D_uX_r  \right \|_{L^p(\Omega; \mathbb{R}^{m \times d})} dr \\
  &  \leq & C \int_s^t e^{-L_1(r-u)} dr \leq C e^{-L_1(s-u)} ( 1 \wedge |t-s|)  \,.
\end{eqnarray*}
Finally we   prove \eqref{ineq4.dext}. Using \eqref{md.dif}, we have the following estimate
\[
  \|D_u X_t - D_v X_t - (D_u X_s - D_v X_s) \|_{L^p(\Omega; \mathbb{R}^{m \times  d})} =\left  \| \int_s^t \sum_{j=1}^l \theta_j \nabla f_j(X_r) (D_uX_r - D_v X_r) dr \right \|_{L^p(\Omega; \mathbb{R}^{m  \times d})} \,.
\]
Applying Minkowski inequality and Cauchy-Schwartz inequality yields
\begin{eqnarray*}
&& \|D_u X_t - D_v X_t - (D_u X_s - D_v X_s) \|_{L^p(\Omega; \mathbb{R}^{m \times  d})} \\
 & \leq & C \int_s^t \|\nabla f_j(X_r) \|_{L^{2p}(\Omega; \mathbb{R}^{m \times m})} \|D_uX_r - D_v X_r\|_{L^{2p}(\Omega; \mathbb{R}^{m  \times  d})} dr \\
& \leq & C (1 \wedge |u-v|)  \int_s^t e^{-L_1(r-u)} dr \leq C e^{-L_1(s-u)}(1\wedge |u-v| ) (1\wedge  |t-s|)  \,.
\end{eqnarray*}
This  proves   \eqref{ineq4.dext} and 
 proof of the proposition is complete.
\end{proof}
\begin{remark} It is worth pointing out that the  solution of the SDE (\ref{nl.sde})   is H\"{o}lder continuous in $L^p$ for all $p\ge 1$ with exponent $H$, i.e.,  
$\|X_t-X_s\|_{L^p(\Omega;\mathbb{R}^m)}\le C|t-s|^H$. However, the Malliavin derivative of $X_t$ is more regular, i.e.,
$ \|D_u X_t - D_u X_s\|_{L^p(\Omega; \mathbb{R}^{m\times   d})} \leq C |t-s|$. That is, the H\"older continuity exponent is improved from $H$ to $1$. This is because the noise in the SDE is additive.
\end{remark}

The next lemma provides bounds for the norm of the derivative of a function of the solution to equation (\ref{nl.sde}).
\begin{lemma}\label{derg.norm}
	Let $H \in (0,\frac{1}{2})$ and $p \geq 2$.  Consider a  function $g= (g^1, \dots, g^d) :\mathbb{R}^m \to \mathbb{R}^d$  whose components  belong
	to  $\mathcal{C}^2_p(\RR^m)$.  Then for all $0\le s\leq t$, we have
		\begin{equation} \label{ddg.norm}
		  \|Dg(X_t) - Dg(X_s)\|_{L^p(\Omega;\mathfrak{H}^{d }\otimes \mathbb{R}^d)} \leq K (t-s)^H  s^\lambda \,,
		\end{equation}
	and
	 \begin{equation}\label{dg.norm}
	   \|Dg(X_s)\|_{L^p(\Omega;\mathfrak{H}^{d } \otimes \mathbb{R}^d)} \leq K s^\lambda \,,
	 \end{equation}
for any $\lambda \in (0, H]$, where $K$ is a constant that may depend on $\lambda$.
\end{lemma}

\begin{proof}
Consider the $\mathfrak{H}^{d } \otimes \R^d $-valued function $\phi:= D g(X_t) - Dg(X_s)$. We can writre
 \begin{eqnarray*}
   \|\phi\|_{\mathfrak{H}^{d }\otimes \R^d}^2 \leq C \|\phi\|_{K^{d}_t\otimes \R^d}^2 & \leq & C \int_0^{  t} |\phi(u)|^2 \left((t-u)^{2H-1} + u^{2H-1}\right) du \\
   && +C \int_0^{  t} \left(\int_v^{  t}|\phi(u)-\phi(v)| (u-v)^{H-\frac{3}{2}} du\right)^2 dv =:C( A_1 + A_2 )\,.
 \end{eqnarray*}	
Therefore,
\[
  \|\phi\|_{L^p(\Omega; \HH^d \otimes \R^d)} \leq  C\sum_{i=1}^2 \|A_i\|_{L^{\frac{p}{2}}(\Omega)}^{\frac{1}{2}} \,.
\]
It remains to  estimate $\|A_i\|_{L^{\frac{p}{2}}(\Omega)}^{\frac{1}{2}} $ for $i=1,2$. First, we write $\phi(u)$ as
 \begin{equation}
    \phi(u)  =  \nabla g(X_t) \cdot (D_uX_t - D_u X_s) + (\nabla g(X_t) - \nabla g(X_s)) \cdot D_u X_s \label{phi.u} \,.
 \end{equation}
Thus, by the submultiplicativity of Hilbert-Schmidt norm, i.e., $|AB| \leq |A||B|$, we have
 \begin{equation*}
 	|\phi(u)| \leq 
 	\begin{cases}
 	|\nabla g(X_t)| |D_uX_t - D_u X_s| + |X_t - X_s| |D_u X_s| \\
 	\qquad\qquad\qquad\qquad\qquad  \times \int_0^1 \|\mathds{H}(g(X_s + r(X_t - X_s)))\| dr&\qquad  \hbox{when}\ u\le s\le t\,;\\
 |\nabla g(X_t)| |D_uX_t| 	&\qquad  \hbox{when}\  s\le u\le   t\,.\\ 
 	\end{cases}
 \end{equation*}
Here $\mathds{H}(g) = (\mathds{H}(g^1), \ldots, \mathds{H}(g^d))$ is understood as the third order tensor, and $\|\mathds{H}(g)\|^2 = \sum_i |\mathds{H}(g^i)|^2$. Since the components of $g$  belong to  $\mathcal{C}^2_p(\RR^m)$, Proposition \ref{p.moment} says that the $L^p$ norm of $|\nabla g(X_t)|$ and $\|\mathds{H}(g(X_t))\|$ are both bounded for any $t\ge 0, p\ge1$. Due to these facts and the inequalities \eqref{ineq.dext} and \eqref{ineq2.dext}, we have
   \begin{eqnarray*}
     \left(\mathbb{E}(|\phi(u)|^p)\right)^{\frac{1}{p}} &\leq& C \begin{cases}
     \left(\mathbb{E}(|\nabla g(X_t)|^{2p})\right)^{\frac{1}{2p}} \left(\mathbb{E}(|D_uX_t - D_u X_s|^{2p})\right)^{\frac{1}{2p}} \\
\qquad   +     e^{-L_1(s-u)}  \int_0^1\left(\mathbb{E}(\|\mathds{H}(g(X_s+r(X_t-X_s)))\|^{2p})\right)^{\frac{1}{2p}}  dr\\
 \qquad  \times \left(\mathbb{E}(|X_t - X_s|^{2p})\right)^{\frac{1}{2p}}&\qquad \hbox{when }\  u\le s \le t;\\
	  \left(\mathbb{E}(|\nabla g(X_t)|^{2p}\right)^{\frac{1}{2p}} \left(\mathbb{E}(|D_uX_t |^{2p})\right)^{\frac{1}{2p}} &\qquad \hbox{when }\    s 
	  \le u\le t\\
	  \end{cases}\\
	  &\leq & C e^{-L_1(s-u)}  (t-s)^H {   \mathds{1}_{\{u<s\}} + C e^{-L_1(t-u)} \mathds{1}_{\{u>s\}}}\,.
   \end{eqnarray*}
Therefore,
\begin{eqnarray*}
	\|A_1\|_{L^{\frac{p}{2}}(\Omega)}^{\frac{1}{2}} & \leq & \left(\int_0^{  t} \left(\mathbb{E}(|\phi|^p)\right)^{\frac{2}{p}} \left((t-u)^{2H-1} + u^{2H-1}\right) du\right)^{\frac{1}{2}} \\
	& \leq & C (t-s)^H \left(\int_0^s e^{- 2 L_1(s-u)} \left((t-u)^{2H-1} + u^{2H-1}\right) du \right)^{\frac{1}{2}} \\
	  &&   + C \left(\int_s^t e^{- 2 L_1(t-u)} \left((t-u)^{2H-1} + u^{2H-1}\right) du \right)^{\frac{1}{2}} \\
	& \leq & C (t-s)^H \,,
\end{eqnarray*}   
where in the last inequality we have used the following arguments. For the second summand, we have bounded $e^{- 2 L_1(t-u)} $ by $1$ and applied  the inequality $t^{2H} - s^{2H} \leq (t-s)^{2H}$.  For the first summand, we bound $(t-u)^{2H-1}$ by $(s-u)^{2H-1}$ and decompose the integral in the intervals $[0,1]$ and $[1,s]$ (if $s\ge 1$).

Now we discuss $A_2$.  For $v < u$, we decompose 
\begin{eqnarray} 
    \phi(u)-\phi(v)   
   &=&\begin{cases}
\left(\nabla g(X_t) - \nabla g(X_s)\right) \cdot (D_uX_t - D_v X_t) \\
\qquad\qquad   + \nabla g(X_s) \cdot \left(D_u X_t - D_v X_t - (D_u X_s - D_v X_s)\right)  &\qquad \hbox{when}\  v<u<s<t\,;\\
    \left(\nabla g(X_t) - \nabla g(X_s)\right) \cdot (D_uX_t - D_v X_t) \\
\qquad\qquad   + \nabla g(X_s) \cdot \left(D_u X_t - D_v X_t + D_v X_s \right)&\qquad \hbox{when}\  v <s<u<t\,;\\
 \nabla g(X_t)  \cdot (D_uX_t - D_v X_t) &\qquad \hbox{when}\  s <v<u<t\,.\\
    \end{cases}  \nonumber\\
  \label{phi.u}  
 \end{eqnarray}  
We shall consider the above three cases separately. 
\newline \noindent  {\it Case 1)}:  $v<u<s<t$.  In this case we have
\begin{eqnarray*}
  \left(\mathbb{E}(|\phi(u) - \phi(v)|^p)\right)^{\frac{1}{p}} & \leq & \int_0^1 \left(\mathbb{E}(\|\mathds{H}(g(X_s+r(X_t-X_s)))\|^{4p})\right)^{\frac{1}{4p}} dr\\
  && \times \left(\mathbb{E}(|X_t - X_s|^{4p})\right)^{\frac{1}{4p}} \left(\mathbb{E}(|D_u X_t - D_v X_t|^{2p})\right)^{\frac{1}{2p}} \\
  & & +  \left(\mathbb{E}(|\nabla g(X_s)|^{2p})\right)^{\frac{1}{2p}} \left(\mathbb{E}(|D_u X_t - D_v X_t - (D_u X_s - D_v X_s)|^{2p})\right)^{\frac{1}{2p}}\,.
\end{eqnarray*}  
 
\noindent {\it Case 2)}:  $s<v<u<t$. We have
\begin{eqnarray*}
   \left(\mathbb{E}(|\phi(u) - \phi(v)|^p)\right)^{\frac{1}{p}} &=& \left(\mathbb{E}( |\nabla g(X_t) \cdot (D_uX_t - D_v X_t) |^p) \right)^{\frac{1}{p}} \\
   &\leq&  \left(\mathbb{E}(|\nabla g(X_t)|^{2p})\right)^{\frac{1}{2p}} \left(\mathbb{E}(|D_uX_t - D_v X_t|^{2p})\right)^{\frac{1}{2p}} \,.
\end{eqnarray*}
\noindent 
{\it Case 3)}:  $v<s<u<t$.  We have
\[
\phi(u) - \phi(v) = \nabla g(X_t) \cdot D_uX_t - \nabla g(X_t) \cdot (D_v X_t - D_v X_s) - (\nabla g(X_t) - \nabla g(X_s)) \cdot D_vX_s \,,
\]
so
\begin{eqnarray*}
  \left(\mathbb{E}(|\phi(u) - \phi(v)|^p)\right)^{\frac{1}{p}} & \leq & \left(\mathbb{E}(|\nabla g(X_t)|^{2p})\right)^{\frac{1}{2p}} \left( \left(\mathbb{E}(|D_u X_t|^{2p})\right)^{\frac{1}{2p}} + \left(\mathbb{E}(|D_v X_t - D_v X_s|^{2p})\right)^{\frac{1}{2p}} \right) \\
  & & +  \int_0^1\left(\mathbb{E}(\|\mathds{H}(g(X_s+r(X_t-X_s)))\|^{4p})\right)^{\frac{1}{4p}}dr \\
  &&\times \left(\mathbb{E}(|X_t - X_s|^{4p})\right)^{\frac{1}{4p}} \left(\mathbb{E}(|D_v X_s|^{2p})\right)^{\frac{1}{2p}} \,.
\end{eqnarray*}
Combining the above cases, and using the inequalities \eqref{ineq.dext} to \eqref{ineq4.dext} in Proposition \ref{p.moment}, we obtain
\begin{eqnarray} 
  \left(\mathbb{E}(|\phi(u) - \phi(v)|^p)\right)^{\frac{1}{p}}  & \leq &    
  C |t-s|^H e^{-L_1(s-u)} |u-v|^\epsilon {  \mathds{1}_{\{v<u<s<t\}}}   + \ Ce^{-L_1(t-u)} |u-v|^\epsilon \mathds{1}_{\{v>s\}} \nonumber \\
    & & \  + \ C \left(e^{-L_1(t-u)} + e^{-L_1(s-v)} |t-s|^H \right) \mathds{1}_{\{v<s<u<t\}} \nonumber \\
	& & \ \  := \sum_{i=1}^4 A_{2i} \,,\label{phiuv.pm} 
\end{eqnarray}
where we have used $ 1 \wedge |u-v| \le C|u-v|^\epsilon$ for any $\epsilon \in [0,1]$ and $ 1 \wedge |t-s| \le C|t-s|^H$. Now we apply Minkowski's inequality to $\|A_2 \|_{L^{\frac{p}{2}}(\Omega)}^{\frac{1}{2}}$ and then an application of \eqref{phiuv.pm} yields
\[
  \|A_2 \|_{L^{\frac{p}{2}}(\Omega)}^{\frac{1}{2}} \leq \left(\int_0^{t} \left(\int_v^{t} \left(\mathbb{E}|\phi(u) - \phi(v)|^p\right)^{\frac{1}{p}}  (u-v)^{H-\frac{3}{2}} du\right)^2 dv\right)^{\frac{1}{2}} \leq \sum_{i=1}^4 A_2^{(i)}\,,
\]
where
\[
  A_2^{(i)} = \left(\int_0^{t} \left(\int_v^{t} A_{2i} (u-v)^{H-\frac{3}{2}} du\right)^2 dv \right)^{\frac{1}{2}} \,.
\]
For $i=1$, fix $\lambda \in (0, H]$ and set $\epsilon = 1-H+\lambda$ for $A_{21}$ in \eqref{phiuv.pm}. In this way, we obtain
\begin{eqnarray*}
	A_2^{(1)} & \leq & C (t-s)^H \left(\int_0^s \left(\int_v^s e^{-L_1(s-u)} (u-v)^{\lambda-\frac{1}{2} } du \right)^2 dv\right)^{\frac{1}{2}} \\
	& & \ \leq  C(t-s)^H \left(\int_0^s (s-v)^{2\lambda - 1} dv\right)^{\frac{1}{2}} \leq C(t-s)^H s^\lambda,
\end{eqnarray*}
where the second inequality follows from the following estimate. For any $\alpha \in (-1, 0)$,
\begin{eqnarray}
	\int_v^s e^{-L_1(s-u)} (u-v)^{\alpha} du & \le & \int_0^{s-v}  e^{-L_1(s-v-x)} x^{\alpha}dx \nonumber \\
	& \le & \int_0^{\frac{s-v}{2}} e^{-L_1(\frac{s-v}{2})} x^{\alpha} dx +  \int_{\frac{s-v}{2}}^{s-v} (\frac{s-v}{2})^{\alpha} e^{-L_1(s-v-x)} dx \nonumber\\
	& & \le C \left(e^{-L_1(\frac{s-v}{2})} (\frac{s-v}{2})^{\alpha + 1} + (\frac{s-v}{2})^{\alpha}\right) \le C(s-v)^{\alpha} \label{int.est2}\,,
\end{eqnarray}
 taking into account the fact that the function $xe^{-L_1 x}$ is bounded on $[0, \infty)$. \\
For $i=2$, choosing $\epsilon =1$, we can write
\[
	A_2^{(2)} \leq C \left(\int_s^t \left(\int_v^t e^{-L_1(t-u)} (u-v)^{H-\frac{1}{2}} du\right)^2dv\right)^{\frac{1}{2}}  
	\] 
Using \eqref{int.est2} by setting $\lambda = H - \frac{1}{2}$,  we have 
\[
  A_2^{(2)} \leq   
  C \left(\int_s^t (t-v)^{2H - 1} dv\right)^{\frac{1}{2}} \leq C(t-s)^H \,.
\]
For $i=3$,
  \begin{eqnarray*}
    A_2^{(3)} &\leq& C \left(\int_0^s \left( \int_s^t e^{-L_1(t-u)} (u-v)^{H-\frac{3}{2}} du\right)^2 dv\right)^{\frac{1}{2}} \\
	 & & \leq C \int_s^t \left(\int_0^s e^{-2L_1(t-u)} (u-v)^{2H-3} dv\right)^{\frac{1}{2}} du \leq C \int_s^t (u-s)^{H-1} du \leq C(t-s)^H \,.
  \end{eqnarray*}
For $i=4$,
   \begin{eqnarray*}
   	 A_2^{(4)} & \leq & C (t-s)^H \left( \int_0^s \left(\int_s^t (u-v)^{H-\frac{3}{2}} du\right)^2 e^{-L_1(s-v)} dv \right)^{\frac{1}{2}} \\
	  & & \leq C (t-s)^H \left(\int_0^s (s-v)^{2H-1} e^{-L_1(s-v)} dv\right)^{\frac{1}{2}} \leq C(t-s)^H \,.
   \end{eqnarray*}
This finishes the proof of \eqref{ddg.norm}. The proof of \eqref{dg.norm} is similar.
\end{proof}

We next apply  Proposition  \ref{p.moment} and Lemma \ref{derg.norm} to deduce the estimate for the $p$-th moment of the divergence integral $Z_{g,t}$ which is defined as
\begin{equation}\label{zgt.def}
Z_{g,t} := \int_0^t g(X_s) dB_s \,,
\end{equation}
where $\{X_t, t \geq 0\}$ is the solution of the SDE \eqref{nl.sde}, and the function $g: \mathbb{R}^m \to \mathbb{R}^d$  satisfies some regularity and growth conditions.

\begin{prop} \label{zt.pnorm}
Let the divergence integral $Z_{g,T}$ be defined by \eqref{zgt.def}.
\begin{enumerate}
 \item If $H \in (\frac{1}{4}, \frac{1}{2})$ and $p \geq 2$, assume that the  components of the function $g: \mathbb{R}^m \to \mathbb{R}^d$ belong to the space $\mathcal{C}^2_p(\RR^m)$. Then we have
    \[
	  \mathbb{E}( |Z_{g,T}|^p)\leq C T^{pH}(1+  T^{p\lambda}) (1+  T^{pH}) \,,
	\]
for any $\lambda \in  (0,H]$,	where $C>0$ is a constant  independent of $T$.
 \item If $H \in (\frac{1}{2}, 1)$, assume that the components of the function $g: \mathbb{R}^m \to \mathbb{R}^d$ belong to the space $\mathcal{C}^1_p(\RR^m)$.   Then for $p > \frac{1}{H}$, we have 
  $$
  \mathbb{E}( |Z_{g,T}|^p) \leq C T^{pH} \,,
  $$
  for all $T > 0$, where $C > 0$  is independent of $T$.
\end{enumerate}
\end{prop}

\begin{proof}
First, for $H \in (\frac 14, \frac{1}{2})$, by  Proposition \ref{p.moment}, the process $\{g(X_t), t \geq 0\}$ satisfies  conditions (i) and (ii) of Hypothesis \ref{hypo.u} with $\beta=H$, which requires $H> \frac{1}{2} - H$, i.e., $H > \frac{1}{4}$. By  \eqref{ddg.norm} and \eqref{dg.norm} of  Lemma  \ref{derg.norm}, $Dg(X_t)$ satisfies  conditions (iii) and (iv) of Hypothesis \ref{hypo.u} with $\beta=H$ and $\lambda \in (0,H]$. By Proposition \ref{div.pmom}, we obtain the result.

Second, for $H \in (\frac{1}{2}, 1)$, applying the results in the preceding Proposition \ref{p.moment}, we get that $g(X_t)$ and $\nabla g(X_t)$ are bounded in $L^p(\Omega)$, so clearly $g(X_t)$ is in the space $\mathbb{D}^{1,p}(\mathfrak{H}^d)$. Applying Lemma \ref{pnorm.g} to $Z_{g,T}$ yields
\[
  \mathbb{E} (|Z_{g,T}|^p) \leq C_{p,H} \left(  \left(\int_0^T \mathbb{E}(|g(X_t)|^{\frac{1}{H}} )dt \right)^{pH} + \mathbb{E} \left( \int_0^T \int_0^t |D_s g(X_t)|^{\frac{1}{H}} ds dt \right)^{pH} \right) \,.
\]
Then we use \eqref{ineq.dext} and integrate $s$ to obtain
\begin{eqnarray*}
  \mathbb{E}( |Z_{g,T}|^p) &\leq& C_{p,H} \left(  \left(\int_0^T \mathbb{E}(|g(X_t)|^{\frac{1}{H}}) dt \right)^{pH} + \frac{|\sigma|^p H^{pH}}{L_1^{pH}} \mathbb{E} \left( \int_0^T  |\nabla g(X_t)|^{\frac{1}{H}} (1-e^{-\frac{L_1}{H}t}) dt \right)^{pH} \right) \\
  &\leq& C_{p,H} \left( \int_0^T \mathbb{E}(|g(X_t)|^{\frac{1}{H}}) dt \right)^{pH} + C_{p,H,L_1,\sigma} \left(\int_0^T \left(\mathbb{E}|\nabla g(X_t)|^p\right)^{\frac{1}{pH}} dt \right)^{pH} \leq  CT^{pH} \,.
\end{eqnarray*}
This concludes the proof.
\end{proof}

\subsection{Proof of Theorem \ref{thm.cons}}

 The following lemma is an important ingredient of the proof of Theorem \ref{thm.cons}.
\begin{lemma}\label{f.erg.pos}
	Suppose $f$ satisfies  $ \mathbb{P} \left(  \det (f^{tr}f)(\overline X) >0 \right) >0$, then $\mathbb{E} \left( (f^{tr}f)(\overline X)  \right)$ is invertible. 
\end{lemma}

\begin{proof} Let $\nu$ be the law of $\overline{X}$.
  Applying Minkowski determinantal inequality and Jensen's inequality yields
   $$\text{det}\left(\int_{\mathbb{R}^m}(f^{tr}f)(x) \nu(dx)\right)^{\frac{1}{l}} \geq \int_{\mathbb{R}^m} \text{det}\left((f^{tr}f)(x)\right)^{{\frac{1}{l}}} \nu(dx)\,,$$	
   which is positive under  our hypothesis. 
\end{proof}

Next we proceed to prove Theorem \ref{thm.cons}. Recall that the estimator $\hat\theta_T$ is given  by \eqref{theta.est}. 
 By Theorem \ref{ergodic}, we have 
\[
\frac{1}{T} \int_0^T (f^{tr}f)(X_t)dt \to \mathbb{E} \left((f^{tr}f)(\overline X)\right)  \quad {\rm a.s.} \,,
\]
which is invertible. Therefore,
\begin{equation}  \label{kl1}
 \left( \frac{1}{T} \int_0^T (f^{tr}f)(X_t)dt \right) ^{-1} \to \left( \mathbb{E} \left((f^{tr}f)(\overline X)\right)  \right)^{-1} \quad {\rm a.s.} \,.
\end{equation}
  Fix $j=1, \dots, l$ and consider the function
$g_j (x)= f_j^{tr}(x) \sigma: \R^m \rightarrow \R^d$.
Denote 
\[
 Z_{ j,t} = \int_0^t  g_j (X_s)  dB_s   = \int_0^t f^{tr} _j (X_s) \sigma dB_s.
 \]
for $j=1, \ldots, l$.  Taking into account (\ref{kl1}), to show 
$\lim_{T\rightarrow \infty}
 \frac 1T   |\hat \theta_T-\theta|=0$  it suffices to show 
\begin{equation} \label{ecua1}
\lim_{T\rightarrow \infty}
 \frac 1T      Z_{j,T}  =0
\end{equation}
 for each $j=1, \ldots, l$.    
 The proof of (\ref{ecua1}) will be done in two steps.

\medskip
\noindent{ \it Step 1}:  Fix $j=1, \dots, l$. We  first show that
\[
\lim_{n \rightarrow \infty}
  n^{-1}  Z_{j,n}  =0.
\]
Since the components of  $f$ belong to the space $\mathcal{C}^i_p(\R^m)$ with $i=1,2$, depending on $H>\frac 12$ or $H<\frac 12$, respectively,   clearly the function $g_j(x) $   satisfies the conditions in Proposition \ref{zt.pnorm}.  Applying Proposition \ref{zt.pnorm}, 
\begin{equation}
   \EE( | Z_{j,n}|^p)    \leq 
   \begin{cases}
   C n^{pH } &\qquad \hbox{when\  $H\in (\frac{1}{2}, 1)$}\\
      Cn^{p(2H +\lambda) } &\qquad \hbox{when\  $H\in (\frac  14, \frac{1}{2})$} \,\\
     \end{cases} 
\label{e.4.14}
\end{equation}
for any $\lambda \in (0,H]$. We will choose $p$ and $\lambda$ in such a way that $p > \frac 1{1-H}$ if $H\in (\frac 12,1)$ and $0<\lambda < 1-2H$ and
$p>\frac 1 {1-2H-\lambda}$ if $H\in (0, \frac 12)$.

 On the other hand,  for any $\epsilon > 0$,  by Chebyshev inequality and the  above estimates we  have 
\begin{eqnarray*} 
 \sum_{n=1}^{\infty} \mathbb{P}(\left| n^{-1} Z_{j,n} \right| > \epsilon) 
 &\leq &   
 \sum_{n=1}^\infty  \epsilon^{-p}\EE\left(\left| n^{-1} Z_{j,n} \right|^p\right)\\
 &\leq &   
\begin{cases}
 C \sum_{n=1}^\infty  \epsilon^{-p} n^{(H-1)p}  &\qquad \hbox{when\  $H\in (\frac{1}{2}, 1)$}\\
C   \sum_{n=1}^\infty  \epsilon^{-p} n^{(2H+\lambda -1)p} &\qquad \hbox{when\  $H\in (0, \frac{1}{2})$}\\
\end{cases}\\
&<& \infty. 
\end{eqnarray*}  
By Borel-Cantelli lemma, $n^{-1} Z_{j,n} \to 0$ a.s. as $n \to \infty$. 

\medskip
\noindent
{\it Step 2:} For any $T>0$ we define the integer $k_T$ by $k_T \le T < k_T+1$. We  write 
\[
  \frac 1T Z_{j,T} =  \frac{k_T}{T} \frac{1}{k_T}\int_0^{k_T} g_j(X_t) dB_t   + \frac{1}{T} \int_{k_T}^T g_j(X_t) dB_t \,.
\]
Thus,
\[
\frac 1T \left|  Z_{j,T} \right|  \le   \frac{1}{k_T}   \left| \int_0^{k_T} g_j(X_t) dB_t   \right| + \frac{1}{T} \left|  \int_{k_T}^T g_j(X_t) dB_t  \right|\,.
\]
Clearly from Step 1 the first   summand  converges to $0$ almost surely as $T \to \infty$. For the second summand, observe that
\begin{equation}\label{eq1.cons}
 \frac 1T \left| \int_{k_T}^T g_j(X_t) dB_t \right|  \leq  \frac{1}{k_T} \sup_{t \in [k_T, k_T+1]} \left|\int_{k_T}^t g_j(X_s) dB_s\right| \,.
\end{equation}
Now we 
apply  Theorem \ref{div.maxineq} to the $p$-th moment of 
$\sup_{t \in [k_T, k_T+1]} \left|\int_{k_T}^t g_j(X_s) dB_s\right| $.  When $H\in (\frac{1}{2}, 1)$, we have  
 \begin{eqnarray*}
\EE\left[  \sup_{t \in [k_T, k_T+1]} \left|\int_{k_T}^t g_j(X_s) dB_s\right| ^p\right]    &\leq& C   \left( \int_{k_T}^{k_T+1} \mathbb{E}(|g_j(X_s)|^p) ds + \int_{k_T}^{k_T+1} \int_{k_T}^s \mathbb{E}(|D_{\mu} g_j(X_s)|^p )d\mu ds\right)  \\
  &\leq& C \int_{k_T}^{k_T+1} \mathbb{E} \left(|g_j(X_s)|^p + | \nabla g_j(X_s)|^p \right) ds \leq C \,.
\end{eqnarray*}
Similarly, for $H \in (\frac{1}{4}, \frac{1}{2})$,  $g_j$ belongs to $\mathcal{C}^2_p(\R^m)$, so  by Lemma \ref{derg.norm} it satisfies Hypothesis \ref{hypo.u}.  Then applying Theorem \ref{div.maxineq} yields
 \[
 \EE\left[  \sup_{t \in [k_T, k_T+1]} \left|\int_{k_T}^t g_j(X_s) dB_s\right| ^p\right]  \leq C (k_T + 1)^{p\lambda}\,
 \]
 for all $p > \frac{1}{H}$, and any $\lambda \in (0, H]$.
By Chebyshev inequality,
\begin{eqnarray*}\label{eq2.cons}
&& \mathbb{P} \left( \frac{1}{k_T} \sup_{t \in [k_T, k_T+1]} \left|\int_{k_T}^t g_j(X_s) dB_s\right| > \epsilon \right) \\
&& \leq  \epsilon^{-p} \EE
\left( \frac{1}{{k_T}^p} \sup_{t \in [k_T, k_T+1]} \left|\int_{k_T}^t g_j(X_s) dB_s\right|^p\right)  \le C \epsilon^{-p}  k_T^{p\lambda-p} \,.
\end{eqnarray*}
Choosing $p$ large enough, the above right-hand side is  summable with respect to $k_T$ and the desired result just follows from Borel-Cantelli Lemma.
This completes the proof of Theorem   \ref{thm.cons}.

 \begin{rem}
 	From the proof of Theorem \ref{thm.cons} we can see that the random variables $\xi_t = t^{-1} Z_{j,t}$ converge to $0$ as $t$ tends to infinity for every $j=1, \ldots, l$ in the following sense. For any $\epsilon>0$,
	$$\lim_{n \to \infty}\sum_{k=n}^\infty \mathbb{P}(\sup_{k \leq t \leq k+1} |\xi_t| > \epsilon) = 0\,.$$ This type of convergence is analogous to the complete convergence of a sequence of random variables (see \cite{hr}), which implies the almost sure convergence.
 \end{rem}
 \begin{rem}
 	If we assume that the parameter vector $\theta$ belongs to a compact set $\Theta \subset \mathbb{R}^l$, the upper bound of the $p$-th moment of  $X_t$  would be independent of $\theta$, and, correspondingly, the constants $C$ and $K$ that appear in Proposition \ref{p.moment}, Lemma \ref{derg.norm} and Proposition \ref{zt.pnorm} would be independent of $\theta$ as well. As a consequence, we get the uniform strong convergence of the random variables $\xi_t = t^{-1} Z_{j,t}$ to $0$ as $t$ tends to infinity for every $j=1, \ldots, l$, in the sense of 
	$$\lim_{T \to \infty} \sup_{\theta \in \Theta} \mathbb{P}( \sup_{t\ge T}|\xi_t| > \epsilon) = 0$$ 
for any $\epsilon > 0$. Furthermore, if the function $f$ satisfies  $(f^{tr}f)^{-1} \leq L_3 I_l$ where $L_3 > 0$ is a constant independent of $\theta$ and $I_l$ is an $l \times l$ identity matrix, the uniform strong consistency of $\hat{\theta}_T$ can be established  by observing that $\left(\frac{1}{T} \int_0^T (f^{tr} f) (X_t) dt \right)^{-1} \leq L_3 I_l$.
 \end{rem}

\section*{Acknowledgement} 
  David Nualart is supported by the NSF grant  DMS1512891.

\begin{minipage}{1.0\textwidth}
\vskip 1cm
Yaozhong Hu: Department of Mathematical and  Statistical Sciences, University of Alberta, Edmonton, Canada, T6G 2G1. {\it E-mail address: yaozhong@ualberta.ca}\\

David Nualart and Hongjuan Zhou: Department of Mathematics, University of Kansas, 405 Snow Hall, Lawrence, Kansas, 66045, USA. 
{\it E-mail address: nualart@ku.edu,  zhj@ku.edu}
\end{minipage}

\end{document}